\documentclass[12pt,oneside]{amsart}
\usepackage[T1]{fontenc}
\usepackage[latin9]{inputenc}
\usepackage{amsthm}
\usepackage{amssymb}

\makeatletter
\numberwithin{equation}{section}
\numberwithin{figure}{section}
\theoremstyle{plain}
\newtheorem{thm}{\protect\theoremname}[section]
  \theoremstyle{remark}
  \newtheorem{rem}[thm]{\protect\remarkname}
  \theoremstyle{plain}
  \newtheorem{cor}[thm]{\protect\corollaryname}
  \theoremstyle{plain}
  \newtheorem{prop}[thm]{\protect\propositionname}
  \theoremstyle{plain}
  \newtheorem{lem}[thm]{\protect\lemmaname}

\usepackage{dsfont}
\usepackage{pgf,tikz}
\usepackage{mathrsfs}

\newtheorem{hyp}{Assumption}

\usetikzlibrary{arrows}

\definecolor{xdxdff}{rgb}{0.49,0.49,1.}
\definecolor{qqqqff}{rgb}{0.,0.,1.}
\definecolor{ffqqqq}{rgb}{1.,0.,0.}

\makeatother

  \providecommand{\corollaryname}{Corollary}
  \providecommand{\lemmaname}{Lemma}
  \providecommand{\propositionname}{Proposition}
  \providecommand{\remarkname}{Remark}
\providecommand{\theoremname}{Theorem}

\begin{document}

\title[H.-S. Representation for a Ginzburg-Landau Process]{A Note on Helffer-Sjöstrand Representation for a Ginzburg-Landau
Process}

\author{Paul de Buyer}
\address[P. de Buyer]{Université Paris Nanterre - Modal'X, 200 avenue
de la République 92000 Nanterre, France}
\email{debuyer@math.cnrs.fr}



\begin{abstract}
In this work, we explore a link between an unbounded spin system and
a random walk. This allows us to study the decay of the (co)variance
of functions with respect to time. We extend here the previous work
of T. Bodineau and G. Graham \cite{bodineauGraham} to a more general
class of graph and potential. 
\end{abstract}

\maketitle

\section{Introduction}

A model of Ginzburg-Landau is a conservative model defined from a
system of stochastic differential equations whose drift is given by
a gradient (in the discrete sense) of a function, called the potential
function. On each vertex $x$ of a graph $G$, we assign a real value
$\eta_{x}$ called a mass which evolves according to the value its
neighbors and random part given by Brownian motion. These kind of
dynamics can be seen as hydrodynamic limit of particle system, see
\cite{spohn2012large}, and in this context has been studied in the
80's with a series of article \cite{FritzHydro,FritzResolvent,FritzHydroLLN,Funaki90,Funaki89}. 

Central limit theorem has been broadly discussed, and we refer the
reader to the Ph.D. thesis of J. Sheriff \cite{sheriff2011central}
which is fully dedicated to this subject. The spectral gap and the
log-Sobolev inequality is also studied in particular in \cite{bodineauGraham,chafai2002glauber,LandimPanizoYau,ledoux2001logarithmic}.

In this work, following the approach of \cite{bodineauGraham,RobertoCancriniHSGL},
we use the Helffer-Sjöstrand representation to link the evolution
of the spin system with a random walk in random dynamic environment.
This representation, see \cite{helffer1999remarks2,helffer1999remarks3,helffer2002semiclassical,helfferSjostrandcorrelation},
relies on the commutation of an operator with a gradient giving rise
to a second operator, called the Witten-Laplacian, which has, under
some assumptions, a probabilistic interpretation. A closely related
model, where this method has been used, is the $\nabla\varphi$ interface
model, in \cite{FunakiGradPhi}, which is the same model as ours in
dimension~$1$. 

Furthermore, the space-time correlation in $\mathbb{Z}^{d}$ between
two masses, in \cite{spohn2012large}, is conjectured to behave in
the following way:
\[
\mathbb{C}ov\left(\eta_{x}\left(0\right);\eta_{y}\left(t\right)\right)\simeq\frac{C_{1}}{t^{d/2}}\exp\left(\frac{-\left|x-y\right|^{2}}{C_{2}t}\right)
\]
where $C_{1}$ and $C_{2}$ are explicit. Using the connection between
time-evolution of the covariance of masses and a random walk should
lead to theses estimates. As a short intuition, one can understand
that the behavior of a random walk is encoded by the behavior of the
diffusion of the mass; therefore, the behavior of the masses should
be encoded by the behavior of a particle (of the masses).

In this work, we solve the difficulty encountered in \cite{bodineauGraham}
while trying to use the Helffer-Sjöstrand representation on general
graph by introducing a site approach. The paper is organized in the
following way. In the second section, we introduce the notations,
the model and the results. In the third section, we give the representation
and prove the main theorem. In the last section, we prove the auxiliary
results. 

\section{Notations and Results}

In this section, we introduce the notations, the model and the result.

Let $G=\left(V,E\right)$ be a connected graph, where $V$ is the
set of vertices and $E$ the set of unoriented edges. We assume that
the degree of each vertex is uniformly bounded by a constant $d\geqslant0$.
Furthermore, we fix a vertex and call it origin, denoted $0$. We
note $d_{G}$ the graph distance and $\left|\cdot\right|$ the distance
of a point to the origin. For practical reasons, we give an arbitrary
orientation to the edges and note $\overrightarrow{B}$ the set of
these edges; we note $\overleftarrow{B}$ the set of reversed oriented
edges, \emph{i.e.} $\overleftarrow{B}=\left\{ \left(y,x\right):\left(x,y\right)\in\overrightarrow{B}\right\} $.
Furthermore, we introduce $B=\overleftarrow{B}\cup\overrightarrow{B}$.
To standardize the notations, we write $e=\left\{ x,y\right\} $ for
an unoriented edges and $b=\left(x,y\right)$ for an oriented one
(belonging to $\overrightarrow{B}$ or $\overleftarrow{B}$). For
two vertices $x$ and $y$, we note $x\sim y$ iff the edge $\left\{ x,y\right\} \in E$.
Moreover, for all $x\in V$, we define $sgn_{x}:B\rightarrow\left\{ -1;0;1\right\} $
the function which assign to an oriented edge $b=\left(y,z\right)$
associates :
\[
sgn_{x}\left(b\right)=\begin{cases}
-1 & \mbox{if }y=x\\
1 & \mbox{if }z=x\\
0 & \mbox{else}
\end{cases}
\]
We call an environment, an element $\eta=\left(\eta_{x}\right)_{x\in V}\in\Omega=\mathbb{R}^{V}$
and introduce the partial order $\geqslant$: for any two environments
$\sigma$ and $\eta$, $\sigma\geqslant\eta$ if and only if $\forall x\in V,\sigma_{x}\geqslant\eta_{x}$.
Thereby, we define the notion of increasing function, meaning that
$f$ is an increasing function if $\forall\eta$, $\sigma$ : 
\[
\sigma\geqslant\eta\Rightarrow f\left(\sigma\right)\geqslant f\left(\eta\right)
\]
Let $\left(V_{x}\right)_{x\in V}$ be a family of functions such that
$\int\exp\left(-V_{x}\left(t\right)\right)dt=1$ which satisfied the
following assumptions:

\begin{hyp}[Mean] \label{hyp:GLIndependance} 
There exists a constant $M$ such that for all $x\in V$, $\int_\mathbb{R} t\times \exp(-V_x(t))dt=M$.
\end{hyp}

\begin{hyp}[Potential $C^2$] \label{hyp:GLHypotheseC2} 
for all $x\in V$, $V_x\in C^2$.
\end{hyp}

\begin{hyp}[Strict convexity] \label{hyp:GLHypotheseConvexite} 
There exists two constants $C_+\geqslant C_- >0$ such that for all $\eta \in \Omega , x\in V$,  $C_+\geqslant V_x''(\eta ) \geqslant C_-$.
\end{hyp}

We define the function $H:\Omega\rightarrow\mathbb{R}$, called the
potential function, given by $\forall\eta,H\left(\eta\right)=\sum_{x}V_{x}\left(\eta\right)$. 

Finally, we introduce operators of partial derivatives for all $x\in V$,$\partial_{x}=\frac{d}{d\eta_{x}}$,
and for all $b=\left(x,y\right)\in B$, $\partial_{b}=\partial_{x}-\partial_{y}$.
The dynamic of Ginzburg-Landau is defined by the following system
of differential equations:
\begin{equation}
d\eta_{x}\left(t\right)=\sum_{b\in\overrightarrow{B}}sgn_{x}\left(b\right)\left(\partial_{b}H\left(\eta\right)dt+\sqrt{2}dB_{b}\left(t\right)\right),x\in S\label{eq:GLEDS}
\end{equation}
where $\left(B_{b}\left(t\right)\right)_{b\in\overrightarrow{B}}$
is a family of independent Brownian motion indexed by the set of oriented
edges. From these SDE, we extract the infinitesimal generator $\mathcal{L}_{e}$
which describes the evolution of an environment. The subscript $e$
has been chosen to signify that the operator describes the evolution
of the $e$nvironment. The generator is given by:
\begin{eqnarray}
\mathcal{L}_{e}f\left(\eta\right) & = & -\frac{1}{2}\sum_{x,y:x\sim y}\left(\partial_{x}-\partial_{y}\right)^{2}f\left(\eta\right)\nonumber \\
 &  & \qquad+\frac{1}{2}\sum_{x,y:x\sim y}\left(\partial_{x}-\partial_{y}\right)H\left(\eta\right)\times\left(\partial_{x}-\partial_{y}\right)f\left(\eta\right)\nonumber \\
 & = & -\sum_{x,y:x\sim y}\partial_{x}\left(\partial_{x}-\partial_{y}\right)f\left(\eta\right)\nonumber \\
 &  & \qquad+\sum_{x,y:x\sim y}\partial_{x}H\left(\eta\right)\times\left(\partial_{x}-\partial_{y}\right)f\left(\eta\right)\label{eq:GLGenerateurdx}\\
 & = & -\sum_{b\in\overrightarrow{B}}\partial_{b}\partial_{b}f\left(\eta\right)+\partial_{b}H\left(\eta\right)\partial_{b}f\left(\eta\right)\nonumber 
\end{eqnarray}
where $f$ is a local function, twice differentiable and with a finite
norm $\left\Vert f\right\Vert =\sum_{x\in S}\left\Vert \partial_{x}f\right\Vert _{\infty}$;
these three assumptions on the functions will always be assumed in
the rest of the article. We note for all generator $L$, the associated
semigroup $\left(P_{t}^{L}\right)=\left(e^{-Lt}\right)_{t\geqslant0}$,
in particular, when there is no ambiguity, we will omit the superscript
for the semigroup associated to the generator $\mathcal{L}_{e}$,
meaning that $\left(P_{t}\right)_{t\geqslant0}\colon=\left(P_{t}^{\mathcal{L}_{e}}\right)_{t\geqslant0}$.
A reversible measure of the Ginzburg-Landau process on a finite graph
is the Gibbs measure $\mu$ given by: 
\[
d\mu\left(\eta\right)=\exp\left(-H\left(\eta\right)\right)d\eta=\prod_{x\in S}\exp\left(-V_{x}\left(\eta_{x}\right)\right)d\eta
\]
This definition can be easily extended in the case of infinite graph
under assumption 1, 2 and 3. For simplicity, we will use the probabilistic
notations, meaning that $\mathbb{P}$ will denote the reversible measure,
$\mathbb{E}$ the associated mean, and $\mathbb{C}ov\left(f;g\right)=\mathbb{E}\left[\left(f-\mathbb{E}\left[f\right]\right)\left(g-\mathbb{E}\left[g\right]\right)\right]$.
Since the model is conservative, note that $\mathbb{P}_{\rho}$ the
law conditioned on a mass density $\rho:=\rho\left(\eta\right)=\frac{1}{\left|V\right|}\sum_{x}\eta_{x}$
(defined for finite graph), is also a reversible measure. 

We define the random walk $\left(X\left(t\right)\right)_{t\geqslant0}$
on the vertices whose jump rates are dependent of the environment
$\left(\eta\left(t\right)\right)_{t\geqslant0}$ which evolve according
to the Ginzburg-Landau process. For all time $t$, the random walk
$X\left(t\right)\in S$ and the jump rate from its position to one
of its neighbor is given by $V_{x}''\left(\eta_{x}\left(t\right)\right)_{t\geqslant0}$.
The law describing the joint evolution of the walker and the environment
is noted $\mathbb{P}_{x}^{\sigma}\left(\cdot\right)=\mathbb{P}\left(.|X\left(0\right)=x;\eta\left(0\right)=\sigma\right)$.
We write $\mathcal{L}_{p}^{\eta}$ the generator of this random walk
is therefore given by:
\[
\mathcal{L}_{p}^{\eta}f\left(x\right)=\sum_{z\sim x}V''\left(\eta_{x}\right)\left(f\left(z\right)-f\left(x\right)\right)\quad\forall\left(x,\eta\right)\in S\times\Omega
\]
The subscript $p$ has been chosen to signify that the generator acts
on the position. Finally, we define the generator $L=Id\otimes\mathcal{L}_{e}+\mathcal{L}_{p}\otimes Id$
describing the joint evolution of the walker and the environment id
acting on the functions $F:S\times\Omega\rightarrow\mathbb{R}$ in
the following way: 
\begin{equation}
LF\left(x,\eta\right)=\left(Id\otimes\mathcal{L}_{e}\right)F\left(x,\eta\right)+\left(\mathcal{L}_{p}\otimes Id\right)F\left(x,\eta\right),\forall\left(x,\eta\right)\in S\times\Omega\label{eq:GLGenerateurHS}
\end{equation}
To simplify the notations, we will write $\mathcal{L}_{e}$ instead
of $Id\otimes\mathcal{L}_{e}$ and $\mathcal{L}_{p}$ instead of $\mathcal{L}_{p}\otimes Id$
so that we can write $L=\mathcal{L}_{e}+\mathcal{L}_{p}$. 

The main theorem of this article is the following:
\begin{thm}
\label{thm:GLThmPrincipal} Let $G=\left(V,E\right)$ be an infinite
graph with bounded degree. Under assumptions $\ref{hyp:GLIndependance}$,
$\ref{hyp:GLHypotheseC2}$ and $\ref{hyp:GLHypotheseConvexite}$,
$\forall x,y\in V$:
\begin{align*}
\mathbb{C}ov\left(\eta_{x};P_{t}\eta_{y}\right) & \leqslant\frac{1}{C_{-}}\mathbb{E}\left[\mathbb{P}_{x}^{\eta}\left(X\left(t\right)=y\right)\right]\\
\mathbb{C}ov\left(\eta_{x};P_{t}\eta_{y}\right) & \geqslant\frac{1}{C_{+}}\mathbb{E}\left[\mathbb{P}_{x}^{\eta}\left(X\left(t\right)=y\right)\right]
\end{align*}
\end{thm}

\begin{rem}
If we have space and time bounds on the coefficients $\mathbb{P}_{x}^{\eta}\left(X\left(t\right)=y\right)$,
then, these bounds apply immediately to the Ginzburg-Landau model.
Note that bounds can be (and some have been obtained in \cite{giacomin2001equilibrium})
on $\mathbb{Z}^{d}$ following \cite{carlen1987upper,mourrat2011variance}.
Furthermore, if $C_{-}=C_{+}$, so that the potential is Gaussian,
we obtain the equality instead of inequality above, and the walk is
a simple random walk. 
\begin{rem}
The uniform bounded degree assumption on the graph may not be necessary
as well as the assumptions 1 and 3. Indeed, they guarantee the well
definition of the model but they can be extended. 
\end{rem}

The main result has the following corollary:
\end{rem}

\begin{cor}
\label{cor:CoroThmPrinc}Under the assumptions $\ref{hyp:GLIndependance}$,
$\ref{hyp:GLHypotheseC2}$ and $\ref{hyp:GLHypotheseConvexite}$,
then for all function $f\colon\Omega\to\mathbb{R}$ such that 
\begin{equation}
\left\Vert f\right\Vert =\sum_{x\in S}\left\Vert \partial_{x}f\right\Vert _{\infty}<+\infty\label{eq:GLcorollaireThPConditionf}
\end{equation}
the following inequality holds:
\[
\mathbb{C}ov\left(f;P_{t}f\right)\leqslant\frac{1}{C_{-}}\left\Vert f\right\Vert ^{2}\sup_{x\in S}\mathbb{E}\left[\mathbb{P}_{x}^{\eta}\left(X_{t}=x\right)\right]
\]
Furthermore, if $f$ and $g$ are two increasing functions then, 
\[
\mathbb{C}ov\left(f;P_{t}g\right)\leqslant\frac{\left\Vert f\right\Vert \times\left\Vert g\right\Vert }{C_{-}}\sup_{\substack{x\in supp\left(f\right)\\
y\in supp\left(g\right)
}
}\mathbb{E}\left[\mathbb{P}_{x}^{\eta}\left(X_{t}=y\right)\right]
\]
\end{cor}

The next proposition, taken from \cite{bodineauGraham}, compares
the spectral gap of the Ginzburg-Landau process with the spectral
gap of the random walk. The proof is given for the sake of completeness.
\begin{prop}
\label{prop:TheSpectralGap}The spectral gap $\lambda_{e}^{*}$ of
the Ginzburg-Landau Process given by the generator $\mathcal{L}_{e}$
is lower bounded by the spectral gap $\lambda_{L}^{*}$ of the random
walk given by the generator $L$:
\[
\lambda_{e}^{*}\geqslant\lambda_{L}^{*}
\]
\end{prop}

The spectral gap has been derived especially in \cite{caputo2004spectral,chafai2002glauber,LandimPanizoYau}.
In particular, in $\left(\mathbb{Z}/N\mathbb{Z}\right)^{d}$, one
can easily obtain a spectral gap of order $C/N^{2}$ with a constant
$C$ that depends on $C_{-}$ and $C_{+}$ using the Helffer-Sjöstrand
representation.

\section{\label{sec:GLHSParSite}The Representation of Helffer-Sjöstrand by
Site}

in this section, we prove start by proving the main theorem using
two key lemmas which will be proved later.

\subsection{Proof of the Theorem $\ref{thm:GLThmPrincipal}$}

The two following key lemmas are inspired by \cite{bodineauGraham}:
\begin{lem}
\label{lem:GLEgaliteMarcheAleatoire}For all $t\geqslant0$ and all
$x,y\in S$, we have the following equality:
\[
\mathbb{C}ov\left(V_{x}'\left(\eta_{x}\right);P_{t}\eta_{y}\right)=\mathbb{E}\left[\mathbb{P}_{x}^{\eta}\left(X\left(t\right)=y\right)\right]
\]
\end{lem}

\begin{lem}
\label{lem:GLComparaison}For all increasing functions $f$ and $g$,
we have the:
\[
\mathbb{C}ov\left(f;P_{t}g\right)\geqslant0
\]
\end{lem}

From these two key lemmas, the proof of the main theorem is immediate:
\begin{proof}[Proof of the theorem $\ref{thm:GLThmPrincipal}$]
 Consider two vertices $x$ and $y$ of $G$ and the functions $f:\Omega\rightarrow\mathbb{R}$
and $g:\Omega\rightarrow\mathbb{R}$ defined by:
\begin{eqnarray*}
f\left(\eta\right) & = & \frac{1}{C_{-}}V'\left(\eta_{x}\right)-\eta_{x}\\
g\left(\eta\right) & = & \eta_{y}
\end{eqnarray*}
Since $f$ and $g$ are increasing, using lemma $\ref{lem:GLComparaison}$,
we obtain:
\[
\mathbb{C}ov\left(f;P_{t}g\right)\geqslant0\Rightarrow\mathbb{C}ov\left(\eta_{x};P_{t}\eta_{y}\right)\leqslant\frac{1}{C_{-}}\mathbb{C}ov\left(V'\left(\eta_{x}\right);P_{t}\eta_{y}\right)
\]
Then, using lemma $\ref{lem:GLEgaliteMarcheAleatoire}$ which concludes
the proof of the first part of the theorem. Using the same reasoning
with $f:\eta\mapsto\eta_{x}-\frac{1}{C_{+}}V'\left(\eta_{x}\right)$,
we obtain the second inequality which concludes the proof of $\ref{thm:GLThmPrincipal}$.
\end{proof}

\subsection{Proof of Lemma $\ref{lem:GLEgaliteMarcheAleatoire}$}

In this section, we introduce the Helffer-Sjöstrand representation,
a probabilistic interpretation of the intertwining technique. We recall
briefly that the intertwining technique is the commutation of a generator
and an operator. Here, we will commute $\mathcal{L}_{e}$ and $\partial_{x}$
so we get a second generator satisfying (informally) the relation
$\partial_{x}\mathcal{L}_{e}=L\partial_{x}$. Recall that $L$ is
the infinitesimal generator describing the joint evolution of the
random walk and the environment. This relation is the key to interpret
the evolution of the mass with the evolution of the random walk. 

The Helffer-Sjöstrand representation is the following:
\begin{lem}
\label{lem:GLEntrelas}Define $G\left(x,\eta\right)=\partial_{x}g\left(\eta\right)$
of some function $g$ We have the following inequality:
\[
\partial_{x}\mathcal{L}_{e}g\left(\eta\right)=LG\left(x,\eta\right)
\]
where $L$ is defined by$\eqref{eq:GLGenerateurHS}$. Consequently:
\[
\partial_{x}P_{t}^{\mathcal{L}_{e}}g\left(\eta\right)=P_{t}^{L}G\left(x,\eta\right)
\]
\end{lem}

\begin{proof}
By definition:
\begin{eqnarray*}
\partial_{x}\mathcal{L}_{e}g\left(\eta\right) & = & -\sum_{b}\partial_{b}\partial_{b}\partial_{x}g\left(\eta\right)+\sum_{b}\partial_{b}H\left(\eta\right)\times\partial_{b}\partial_{x}g\left(\eta\right)\\
 &  & +\sum_{b}\partial_{b}\partial_{x}H\left(\eta\right)\times\partial_{b}g\left(\eta\right)\\
 & = & \mathcal{L}_{e}\partial_{x}g\left(\eta\right)+\sum_{y\sim x}V''\left(\eta_{x}\right)\times\left(\partial_{x}g\left(\eta\right)-\partial_{y}g\left(\eta\right)\right)\\
 & = & \left(\mathcal{L}_{e}+\mathcal{L}_{p}\right)\partial_{x}g\left(\eta\right)=LG\left(x,\eta\right)
\end{eqnarray*}
The consequence is due to the fact $P_{t}^{L}=\exp\left(-Lt\right)$.
\end{proof}
We begin the proof of the lemma by recalling the formula of integration
by part. :

\begin{equation}
\mathbb{E}\left[f\times\mathcal{L}_{e}g\right]=\mathbb{E}\left[\sum_{x,y:x\sim y}\partial_{x}f\times\left(\partial_{x}-\partial_{y}\right)g\right]\label{eq:GLIPP}
\end{equation}
Applying this formula with the functions $f\colon\eta\to V_{x}'\left(\eta_{x}\right)$
and $g\colon\eta\to\eta_{y}-\mathbb{E}\left[\eta_{y}\right]$, defining
the function $G:\left(x,\eta\right)\to\partial_{x}g\left(\eta\right)$
and noting that for all $h:\Omega\mapsto\mathbb{R}$ we have $\mathbb{E}\left[\mathcal{L}_{e}h\right]=0$,
we obtain: 
\begin{align*}
\partial_{t}\mathbb{E}\left[f\times P_{t}g\right] & =\mathbb{E}\left[\sum_{y:x\sim y}V_{x}''\left(\eta_{x}\right)\times\left(\partial_{x}-\partial_{y}\right)P_{t}g\right]\\
 & =\mathbb{E}\left[\sum_{y:x\sim y}V_{x}''\left(\eta_{x}\right)\times\left(P_{t}^{L}G\left(x,\eta\right)-P_{t}^{L}G\left(y,\eta\right)\right)\right]\\
 & =\mathbb{E}\left[\mathcal{L}_{p}P_{t}^{L}G\left(x,\eta\right)\right]\\
 & =\mathbb{E}\left[LP_{t}^{L}G\left(x,\eta\right)\right]=\partial_{t}\mathbb{E}\left[P_{t}^{L}G\left(x,\eta\right)\right]
\end{align*}
Which immediately implies that $\mathbb{E}\left[f\times P_{t}g\right]=\mathbb{E}\left[P_{t}^{L}G\left(x,\eta\right)\right]$.
By the definition of $G$, we have for all $z$ and $\eta$ that $G\left(z,\eta\right)=\mathds{1}_{y=z}$,
therefore $P_{t}^{L}G\left(x,\eta\right)=\mathbb{P}_{x}^{\eta}\left(X\left(t\right)=y\right)$
which concludes the proof of the lemma.  \begin{flushright} $\square$ \end{flushright}

\subsection{Proof of Lemma $\ref{lem:GLComparaison}$}

In this section, we prove Lemma $\ref{lem:GLComparaison}$. Given
the independence assumption on the potential ($H\left(\eta\right)=\sum V_{x}\left(\eta_{x}\right)$),
the Harris FKG-inequality would be enough if one can prove that the
partial order is preserved with time. Indeed, in this case, if $g$
is an increasing function, then $P_{t}g$ is an increasing function
and therefore the lemma is proved. This lemma is already proven in
\cite[Lemma 4.8]{bodineauGraham}, so we give here a improved version
on a more general class of potential for the time preservation of
the partial order following the same ideas.

In the next lemma, we consider that the potential there exists two
families of functions of $C^{2}\left(\mathbb{R},\mathbb{R}\right)$
$\left(V_{x}\right)_{x\in S}$ and $\left(V_{x,y}\right)_{x,y\in V,y\sim x}$
such that $H\left(\eta\right)=\sum_{x,y\sim x}V_{x}\left(\eta_{x}\right)+V_{x,y}\left(\eta_{x}+\eta_{y}\right)$.
\begin{lem}
\label{lem:GLConservationDeLOrdre}Assume that there exists four constants
$C_{1,-},C_{2,-},C_{1,+}$ and $C_{2,+}$ such that:
\begin{eqnarray*}
C_{1,-}\leqslant V''_{x}\leqslant C_{1,+}\\
C_{2,-}\leqslant V''_{x,y}\leqslant C_{2,+}
\end{eqnarray*}
Then the partial order is time preserved if:
\begin{eqnarray*}
\inf_{x}\inf_{y\sim x}\left[\sum_{z\sim y}C_{2,-}\right]-C_{2,+}+C_{1,-} & \geqslant & 0\\
C_{2,-} & \geqslant & 0
\end{eqnarray*}
\end{lem}

\begin{proof}
We define two solutions $\left(\eta\left(t\right)\right)_{t\geqslant0}$
and $\left(\sigma\left(t\right)\right)_{t\geqslant0}$ of the system
of SDE $\eqref{eq:GLEDS}$ driven by the same family of Brownian motions.and
such that $\eta\left(0\right)\geqslant\sigma\left(0\right)$. We define
for all $x\in V$ the function $\phi_{x}:\mathbb{R}_{+}\rightarrow\mathbb{R}$,
defined by $\phi_{x}\left(t\right)=\eta_{x}\left(t\right)-\sigma_{x}\left(t\right)$.
Of course, this is a continuous function $\phi_{x}$. Then, we define
the function $\Phi:\mathbb{R}_{+}\rightarrow\mathbb{R}$ by $\Phi\left(t\right)=\sum_{x\in S}\left(2d\right)^{-|x|}\phi_{x}^{2}\left(t\right)\mathds{1}_{\left\{ \phi_{x}\left(t\right)<0\right\} }$
where $d$ is the uniform upper bound on the degree of the vertices
of $V$. Note that $\Phi$ is continuous, positive, $\Phi\left(0\right)=0$
and well defined, see \cite[Théorème 2.1]{shiga1980infinite}. Intuitively,
this function $\Phi$ measure the number of negative $\phi_{x}$.
Using Gronwall Lemma, we will show that the function $\Phi$ is equal
to the zero constant function.

To simplify the notations, we will omit to write the dependence on
$t$ of the processes $\eta\left(t\right)$ and $\sigma\left(t\right)$
when there are no ambiguity and rather write $\eta$ and $\sigma$.
In the same way, we will write $\eta_{x}$ and $\sigma_{x}$ instead
$\eta_{x}\left(t\right)$ and $\sigma_{x}\left(t\right)$. 

To establish the comparison lemma, we differentiate each summand of
$\Phi$ :
\begin{align*}
d\phi_{x}\left(t\right)= & d\left(\eta_{x}\left(t\right)-\sigma_{x}\left(t\right)\right)\\
= & \sum_{y\sim x}\left[\sum_{z\sim y}V_{y,z}^{'}\left(\eta_{y}+\eta_{z}\right)-V_{y,z}^{'}\left(\sigma_{y}+\sigma_{z}\right)\right]\\
 & \qquad-\left(V_{x,y}^{'}\left(\eta_{x}+\eta_{y}\right)-V_{x,y}^{'}\left(\sigma_{x}+\sigma_{y}\right)\right)\\
 & +\sum_{y\sim x}V_{y}^{'}\left(\eta_{y}\right)-V_{y}^{'}\left(\sigma_{y}\right)-\left(V_{x}^{'}\left(\eta_{x}\right)-V_{x}^{'}\left(\sigma_{x}\right)\right)
\end{align*}
Consider the case $\phi_{x}\left(t\right)<0$, meaning that $\eta_{x}\left(t\right)\leqslant\sigma_{x}\left(t\right)$.
Consequently, $d\phi_{x}\left(t\right)$ is an increasing function
in$\sigma_{x}$ so we can bound it from below by setting $\sigma_{x}=\eta_{x}$.
For each term, we measure its contribution in term of $\phi_{.}\left(t\right)$.
Observe that:
\begin{align*}
V'_{x,y}\left(\eta_{x}+\eta_{y}\right)-V'_{x,y}\left(\eta_{x}+\sigma_{y}\right)= & \int_{\sigma_{y}}^{\eta_{y}}V''_{x,y}\left(\eta_{x}+s\right)ds\\
V'_{y,z}\left(\eta_{y}+\eta_{z}\right)-V'_{y,z}\left(\sigma_{y}+\sigma_{z}\right)= & V'_{y,z}\left(\eta_{y}+\eta_{z}\right)-V'_{y,z}\left(\sigma_{y}+\eta_{z}\right)\\
 & \quad+V'_{y,z}\left(\sigma_{y}+\eta_{z}\right)-V'_{y,z}\left(\sigma_{y}+\sigma_{z}\right)\\
= & \int_{\sigma_{y}}^{\eta_{y}}V''_{y,z}\left(s+\eta_{z}\right)ds+\int_{\sigma_{z}}^{\eta_{z}}V''_{y,z}\left(\sigma_{y}+s\right)ds\\
V'_{y}\left(\eta_{y}\right)-V'_{y}\left(\sigma_{y}\right)= & \int_{\sigma_{y}}^{\eta_{y}}V''_{y}\left(s\right)ds
\end{align*}
Thus:
\begin{align*}
d\phi_{x}\left(t\right)\geqslant & \sum_{y\sim x}\int_{\sigma_{y}}^{\eta_{y}}\left[\sum_{z\sim y}V''_{y,z}\left(s+\eta_{z}\right)\right]-V''_{x,y}\left(\eta_{x}+s\right)+V''_{y}\left(s\right)ds\\
 & +\sum_{y\sim x}\sum_{z\sim y}\int_{\sigma_{z}}^{\eta_{z}}V''_{y,z}\left(\sigma_{y}+s\right)ds
\end{align*}
Under the assumption of the lemma, the signs of the integrands are
positive. Under the ellipticity assumptions on the potential, we get:
\begin{align*}
d\phi_{x}\left(t\right)\geqslant & \sum_{y\sim x}\phi_{y}\left(t\right)\mathds{1}{}_{\phi_{y}\left(t\right)<0}\left(\left[\sum_{z\sim y}C_{2,+}\right]-C_{2,-}+C_{1,+}\right)\\
 & +\sum_{y\sim x}\sum_{z\sim y}\phi_{z}\left(t\right)\mathds{1}{}_{\phi_{z}\left(t\right)<0}C_{2,+}
\end{align*}
Recall that $\Phi$ is given by: 
\[
\Phi\left(t\right)=\sum_{x\in V}2^{-|x|}\phi_{x}^{2}\left(t\right)\mathds{1}_{\left\{ \phi_{x}\left(t\right)<0\right\} }
\]
By differentiating and writing $C_{max}=\max_{x,y:y\sim x}\left[\sum_{z\sim y}C_{2,+}\right]-C_{2,-}+C_{1,+}$,
we obtain:
\begin{align*}
\frac{d}{dt}\Phi\left(t\right) & \leqslant\sum_{x\in V}2^{1-\left|x\right|}\phi_{x}\left(t\right)\times\mathds{1}_{\phi_{x}\left(t\right)<0}\\
 & \qquad\times C_{max}\left(\sum_{y\sim x}\phi_{y}\left(t\right)\mathds{1}{}_{\phi_{y}\left(t\right)<0}+\sum_{z\sim y}\phi_{z}\left(t\right)\mathds{1}{}_{\phi_{z}\left(t\right)<0}\right)
\end{align*}
Using the trivial inequality $\forall a,b\in\mathbb{R}$, $2ab\leqslant a^{2}+b^{2}$,
then there exists a constant $C:=C\left(C_{max},d\right)$ such that:
\begin{align*}
\frac{d}{dt}\Phi\left(t\right) & \leqslant C\times\sum_{x\in V}\left(2d\right)^{-\left|x\right|}\phi_{x}^{2}\left(t\right)\mathds{1}_{\phi_{x}\left(t\right)<0}\\
 & \leqslant C\times\Phi\left(t\right)
\end{align*}
Finally, using Gronwall Lemma, we obtain:
\[
\Phi\left(t\right)\leqslant\Phi\left(0\right)\times\exp\left(C\times t\right)
\]
Which concludes the proof of the lemma since $\Phi\left(0\right)=0$.
\end{proof}
We can now prove the second key Lemma.
\begin{proof}[Proof of Lemma $\ref{lem:GLComparaison}$]
 As written above, It is enough to show that the Harris-FKG inequality
holds. It relies on the sufficient condition of the theorem 3 of \cite{preston1974generalization}.
For any two configurations $\sigma$ and $\eta$, noting $\left(\eta\lor\sigma\right)_{x}=\max\left\{ \eta_{x},\sigma_{x}\right\} $
and $\left(\eta\land\sigma\right)_{x}=\min\left\{ \eta_{x},\sigma_{x}\right\} $,
if 
\begin{equation}
\mu\left(\eta\lor\sigma\right)\mu\left(\eta\land\sigma\right)\geqslant\mu\left(\sigma\right)\mu\left(\eta\right)\label{eq:GLCritere}
\end{equation}
then the Harris-FKG inequality holds. This criterion if immediately
verified since we are in the independent case, \emph{i.e.} 
\begin{eqnarray*}
V_{x}\left(\left(\eta\lor\sigma\right)_{x}\right)+V_{x}\left(\left(\eta\land\sigma\right)_{x}\right) & = & V_{x}\left(\eta_{x}\right)+V_{x}\left(\sigma_{x}\right),\quad\forall x\in V
\end{eqnarray*}
Thus, the inequality $\eqref{eq:GLCritere}$ is an equality in our
case, which concludes the proof of the lemma.
\end{proof}
\begin{rem}
\label{rem:GLCritereHolley}One can easily verified that the inequality
$\eqref{eq:GLCritere}$ is not true when $H$ can be written $H\left(\eta\right)=\sum_{x}V_{x}\left(\eta_{x}\right)+\sum_{y\sim x}V_{x,y}\left(\eta_{x}+\eta_{y}\right)$,
where $V_{x,y}$ is a strictly convex function. Moreover, in this
case, one can show that $\mathbb{E}\left[\eta_{x}\eta_{y}\right]<0$
for $x\sim y$ and therefore the Harris-FKG can't be established.
\end{rem}

\section{Proof of Auxiliary Results}

We start by proving corollary $\ref{cor:CoroThmPrinc}$:
\begin{proof}
Because of the condition $\left(\ref{eq:GLcorollaireThPConditionf}\right)$
on $f$, we have that $f$ is Lipschitz in all its coordinate. We
write $L_{f}\left(x\right)=\left\Vert \partial_{x}f\right\Vert _{\infty}$
to get:
\[
\left|f\left(\eta\right)-f\left(\eta^{x}\right)\right|\leqslant L_{f}\left(x\right)\left|\eta_{x}-\eta_{x}^{x}\right|
\]
where $\eta^{x}$ is the configuration which all coordinates are equal
to the ones of $\eta$ except (potentially) in $x$,\emph{ i.e}. $\forall y\neq x$
we have $\eta_{y}^{x}=\eta_{y}$. From this, we obtain for the functions
$g_{+}$ and $g_{-}$ defined by $g_{+}\left(\eta\right)=\sum_{x}L_{f}\left(x\right)\eta_{x}+f\left(\eta\right)$
and $g_{-}\left(\eta\right)=\sum_{x}L_{f}\left(x\right)\eta_{x}-f\left(\eta\right)$
are increasing functions and using lemma $\ref{lem:GLComparaison}$:
\[
\mathbb{C}ov\left(P_{t}g_{-};g_{+}\right)\geqslant0\Leftrightarrow\mathbb{C}ov\left(P_{t}\sum_{x}L_{f}\left(x\right)\eta_{x};\sum_{y}L_{f}\left(y\right)\eta_{y}\right)\geqslant\mathbb{C}ov\left(P_{t}f;f\right)
\]
which gives by the main theorem:
\begin{align*}
\mathbb{C}ov\left(f;P_{t}f\right) & \leqslant\sum_{x,y}L_{f}\left(x\right)L_{f}\left(y\right)\mathbb{C}ov\left(\eta_{x};P_{t}\eta_{y}\right)\\
 & \leqslant\sum_{x,y}L_{f}\left(x\right)L_{f}\left(y\right)\sup_{x}\mathbb{C}ov\left(\eta_{x};P_{t}\eta_{x}\right)
\end{align*}
Using Cauchy-Schwartz inequality and concavity of the square root
function, this proves the first part of the corollary. If $f$ and
$g$ are two increasing functions, by lemma $\ref{lem:GLComparaison}$:
\begin{eqnarray*}
\mathbb{C}ov\left(\sum_{x}L_{f}\left(x\right)\eta_{x};P_{t}\sum_{y}L_{g}\left(y\right)\eta_{y}\right) & \geqslant & \mathbb{C}ov\left(f;P_{t}\sum_{y}L_{g}\left(y\right)V_{y}'\left(\eta_{y}\right)\right)\\
 & \geqslant & \mathbb{C}ov\left(f;P_{t}g\right)
\end{eqnarray*}
Expanding the sum, and taking the supremum over the $x$ and $y$
such that $L_{f}\left(x\right)\neq0$ and $L_{g}\left(y\right)\neq0$
and the main theorem, concludes the corollary.
\end{proof}
We then prove proposition $\ref{prop:TheSpectralGap}$.
\begin{proof}
Recall that $\lambda_{L}^{*}$ is the spectral gap of the random walk
defined by: 
\[
\lambda_{L}^{*}=\inf_{f\neq0}\frac{-\partial_{t=0}\Vert P_{t}^{L}f\Vert_{2}^{2}}{\Vert f\Vert_{2}^{2}}
\]
where $\Vert f\Vert_{2}^{2}=\sum f^{2}\left(x\right)$. Therefore
$\Vert P_{t}^{L}f\Vert_{2}^{2}\leqslant\Vert f\Vert_{2}^{2}e^{-\lambda_{L}^{*}t}$,
for all function $f$. In this way,
\begin{align*}
\mathbb{V}ar\left(P_{t}^{L}f\right) & =\int_{\mathbb{R}^{+}}-\partial_{s}\mathbb{V}ar\left(P_{s}f\right)ds\\
 & =\int_{\mathbb{R}^{+}}\mathbb{E}\left[\sum_{x,y}\left(\left(\partial_{x}-\partial_{y}\right)P_{s}f\right)^{2}\right]ds\\
 & =\int_{\mathbb{R}^{+}}\mathbb{E}\left[\sum_{x,y}\left(P_{s}^{L}\left(\partial_{x}-\partial_{y}\right)f\right)^{2}\right]ds\\
 & \leqslant\int_{\mathbb{R}^{+}}e^{-\lambda_{L}^{*}s}\mathbb{E}\left[\sum_{x,y}\left(\left(\partial_{x}-\partial_{y}\right)f\right)^{2}\right]ds\\
 & \leqslant\left(\lambda_{L}^{*}\right)^{-1}\times\left(-\partial_{t}\mathbb{V}ar\left(P_{t}f\right)\right)
\end{align*}
Taking the infimum over the function $f$, one obtains:
\[
\lambda_{L}^{*}\leqslant\inf_{f}\frac{-\partial_{t=0}\mathbb{V}ar\left(P_{t}f\right)}{\mathbb{V}ar\left(f\right)}=\lambda_{e}^{*}
\]
Which ends the proof.
\end{proof}

\section*{Appendix A. Helffer-Sjöstrand Representation by edge}

In the appendix, we develop an extension of the original approach
of \cite{bodineauGraham}, where the authors commuted the generator
$\mathcal{L}_{e}$ and $\partial_{b}$ with $b\in\overrightarrow{B}$.
We give the approach, its difficulties and some associated results.
A discussion of this approach can be found in the Ph.D. thesis of
the author, see \cite{deBuyerThese} (in french). Commuting $\mathcal{L}_{e}$
and $\partial_{b}$, one gets: 
\begin{align}
\partial_{b'}\mathcal{L}_{e}f & =\partial_{b'}(\sum_{b\in\overrightarrow{B}}-\partial_{b}\partial_{b}f+\partial_{b}H\times\partial_{b}f)\nonumber \\
 & =\sum_{b\in\overrightarrow{B}}-\partial_{b}\partial_{b}\partial_{b'}f+\partial_{b}\partial_{b'}H\times\partial_{b}f+\partial_{b}H\times\partial_{b}\partial_{b'}f\nonumber \\
 & =\mathcal{L}_{e}\partial_{b'}f+\sum_{b\in\overrightarrow{B}}\partial_{b}\partial_{b'}H\times\partial_{b}f\label{eq:commutationAreteE0}
\end{align}
The authors argued that the second summand of $\eqref{eq:commutationAreteE0}$
can only be interpreted as the (positive) generator of a random walk
on the directed edges (of $\overrightarrow{B}$) on the torus (or
$\mathbb{Z}$) since the Hessian of the potential function would have
non negative off-diagonal term and positive diagonal term. However,
noting for every $b=\left(x,y\right)\in B$ the edge with a reversed
orientation $\overleftarrow{b}=\left(y,x\right)$, one can note that
$\partial_{b}=-\partial_{\overleftarrow{b}}$ so we get:
\begin{align}
\sum_{b\in\overrightarrow{B}}\partial_{b}\partial_{b'}H\times\partial_{b}f & =\sum_{\substack{b\in B\\
b\sim b'
}
}\partial_{b}\partial_{b'}H\left(\partial_{b}f-\partial_{b'}f\right)\nonumber \\
 & \qquad\qquad+(\partial_{b'}\partial_{b'}H-\sum_{\substack{b\in B\\
b\sim b'
}
}\partial_{b}\partial_{b'}H)\partial_{b}f\nonumber \\
 & =\mathcal{L}'_{p}\partial_{b'}f+(\partial_{b'}\partial_{b'}H-\sum_{\substack{b\in B\\
b\sim b'
}
}\partial_{b}\partial_{b'}H)\partial_{b'}f\label{eq:DefGenArete}
\end{align}
Where $b\sim b'$ means $b'\neq\overleftarrow{b}$ and$\partial_{b}\partial_{b'}H<0$.
On $\mathbb{Z}$, note that $(\partial_{b'}\partial_{b'}H-\sum_{\substack{b'\sim b\in B}
}\partial_{b}\partial_{b'}H)=0$. Furthermore one can see that $\mathcal{L}'_{p}$ is the generator
of a random walk on the set of oriented edges $B$. To continue the
discussion, we can restrict ourselves to the case where the graph
is $\mathbb{Z}^{2}$ and the Gaussian potential ``$H\left(\eta\right)=\frac{1}{2}\sum\eta_{x}^{2}$''.
In this case, the walker follows the law of a simple random walk on
the oriented ``kite'' graph, drawn with unoriented edges in figure
$\ref{fig:kite}$.

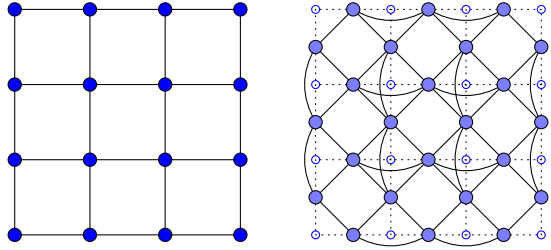
\begin{figure}[!h]
\begin{center}
\begin{tikzpicture}[line cap=round,line join=round,>=triangle 45,x=1.0cm,y=1.0cm] 
\clip(0.8,0.8) rectangle (8.2,4.2); 
\draw (1.,4.)-- (1.,3.); 
\draw (1.,3.)-- (1.,2.); 
\draw (1.,2.)-- (1.,1.); 
\draw (1.,1.)-- (2.,1.); 
\draw (2.,1.)-- (2.,2.); 
\draw (2.,2.)-- (1.,2.); 
\draw (2.,2.)-- (2.,3.); 
\draw (2.,3.)-- (1.,3.); 
\draw (2.,3.)-- (2.,4.); 
\draw (2.,4.)-- (1.,4.); 
\draw (2.,4.)-- (3.,4.); 
\draw (3.,4.)-- (3.,3.); 
\draw (3.,3.)-- (2.,3.); 
\draw (3.,3.)-- (3.,2.); 
\draw (3.,2.)-- (2.,2.); 
\draw (3.,2.)-- (3.,1.); 
\draw (3.,1.)-- (2.,1.); 
\draw (3.,1.)-- (4.,1.); 
\draw (4.,1.)-- (4.,2.); 
\draw (4.,2.)-- (3.,2.); 
\draw (4.,2.)-- (4.,3.); 
\draw (4.,3.)-- (3.,3.); 
\draw (4.,3.)-- (4.,4.); 
\draw (4.,4.)-- (3.,4.); 
\draw [dotted] (5.,4.)-- (5.,3.); 
\draw [dotted] (5.,3.)-- (5.,2.); 
\draw [dotted] (5.,2.)-- (5.,1.); 
\draw [dotted] (5.,1.)-- (6.,1.); 
\draw [dotted] (6.,1.)-- (6.,2.); 
\draw [dotted] (6.,2.)-- (5.,2.); 
\draw [dotted] (6.,2.)-- (6.,3.); 
\draw [dotted] (6.,3.)-- (5.,3.); 
\draw [dotted] (6.,3.)-- (6.,4.); 
\draw [dotted] (6.,4.)-- (5.,4.); 
\draw [dotted] (6.,4.)-- (7.,4.); 
\draw [dotted] (7.,4.)-- (7.,3.); 
\draw [dotted] (7.,3.)-- (6.,3.); 
\draw [dotted] (7.,3.)-- (7.,2.); 
\draw [dotted] (7.,2.)-- (6.,2.); 
\draw [dotted] (7.,2.)-- (7.,1.); 
\draw [dotted] (7.,1.)-- (6.,1.); 
\draw [dotted] (7.,1.)-- (8.,1.); 
\draw [dotted] (8.,1.)-- (8.,2.); 
\draw [dotted] (8.,2.)-- (7.,2.); 
\draw [dotted] (8.,2.)-- (8.,3.); 
\draw [dotted] (8.,3.)-- (7.,3.); 
\draw [dotted] (8.,3.)-- (8.,4.); 
\draw [dotted] (8.,4.)-- (7.,4.); 
\draw (5.5,4.)-- (5.,3.5); 
\draw (5.,3.5)-- (5.5,3.); 
\draw (5.5,3.)-- (6.,3.5); 
\draw (6.,3.5)-- (5.5,4.); 
\draw (6.,3.5)-- (6.5,4.); 
\draw (6.5,4.)-- (7.,3.5); 
\draw (7.,3.5)-- (6.5,3.); 
\draw (6.5,3.)-- (6.,3.5); 
\draw (6.5,3.)-- (6.,2.5); 
\draw (6.,2.5)-- (5.5,3.); 
\draw (5.5,3.)-- (5.,2.5); 
\draw (5.,2.5)-- (5.5,2.); 
\draw (5.5,2.)-- (6.,2.5); 
\draw (6.,2.5)-- (6.5,2.); 
\draw (6.5,2.)-- (6.,1.5); 
\draw (6.,1.5)-- (5.5,2.); 
\draw (5.5,2.)-- (5.,1.5); 
\draw (5.,1.5)-- (5.5,1.); 
\draw (5.5,1.)-- (6.,1.5); 
\draw (6.,1.5)-- (6.5,1.); 
\draw (6.5,1.)-- (7.,1.5); 
\draw (7.,1.5)-- (6.5,2.); 
\draw (6.5,2.)-- (7.,2.5); 
\draw (7.,2.5)-- (6.5,3.); 
\draw (7.,1.5)-- (7.5,1.); 
\draw (7.5,1.)-- (8.,1.5); 
\draw (8.,1.5)-- (7.5,2.); 
\draw (7.,1.5)-- (7.5,2.); 
\draw (7.5,2.)-- (7.,2.5); 
\draw (7.,2.5)-- (7.5,3.); 
\draw (7.5,3.)-- (8.,2.5); 
\draw (8.,2.5)-- (7.5,2.); 
\draw (7.5,3.)-- (7.,3.5); 
\draw (7.,3.5)-- (7.5,4.); 
\draw (7.5,4.)-- (8.,3.5); 
\draw (8.,3.5)-- (7.5,3.); 
\draw (5.5,1.) to [bend right] (6.5,1.); 
\draw (5.5,2.) to [bend right] (6.5,2.); 
\draw (6.,2.5) to [bend right] (6.,1.5); 
\draw (6.,3.5) to [bend right] (6.,2.5); 
\draw (5.5,3.) to [bend right] (6.5,3.); 
\draw (6.5,3.) to [bend right] (7.5,3.); 
\draw (7.,3.5) to [bend right] (7.,2.5); 
\draw (7.,2.5) to [bend right] (7.,1.5); 
\draw (6.5,2.) to [bend right] (7.5,2.); 
\draw (8.,2.5) to [bend right] (8.,1.5); 
\draw (8.,3.5) to [bend right] (8.,2.5); 
\draw (6.5,4.) to [bend right] (7.5,4.); 
\draw (5.5,4.) to [bend right] (6.5,4.); 
\draw (5.,3.5) to [bend right] (5.,2.5); 
\draw (5.,2.5) to [bend right] (5.,1.5); 
\draw (6.5,1.) to [bend right] (7.5,1.); 

\begin{scriptsize} 
	\draw [fill=qqqqff] (1.,1.) circle (2.5pt); 
	\draw [fill=qqqqff] (1.,2.) circle (2.5pt); 
	\draw [fill=qqqqff] (2.,1.) circle (2.5pt); 
	\draw [fill=qqqqff] (2.,2.) circle (2.5pt); 
	\draw [fill=qqqqff] (3.,1.) circle (2.5pt); 
	\draw [fill=qqqqff] (3.,2.) circle (2.5pt); 
	\draw [fill=qqqqff] (4.,1.) circle (2.5pt); 
	\draw [fill=qqqqff] (4.,2.) circle (2.5pt); 
	\draw [fill=qqqqff] (4.,3.) circle (2.5pt); 
	\draw [fill=qqqqff] (3.,3.) circle (2.5pt); 
	\draw [fill=qqqqff] (2.,3.) circle (2.5pt); 
	\draw [fill=qqqqff] (1.,3.) circle (2.5pt); 
	\draw [fill=qqqqff] (1.,4.) circle (2.5pt); 
	\draw [fill=qqqqff] (2.,4.) circle (2.5pt); 
	\draw [fill=qqqqff] (3.,4.) circle (2.5pt); 
	\draw [fill=qqqqff] (4.,4.) circle (2.5pt); 
	\draw [color=qqqqff] (5.,1.) circle (1.5pt); 
	\draw [color=qqqqff] (5.,2.) circle (1.5pt); 
	\draw [color=qqqqff] (6.,1.) circle (1.5pt); 
	\draw [color=qqqqff] (6.,2.) circle (1.5pt); 
	\draw [color=qqqqff] (7.,1.) circle (1.5pt); 
	\draw [color=qqqqff] (7.,2.) circle (1.5pt); 
	\draw [color=qqqqff] (8.,1.) circle (1.5pt); 
	\draw [color=qqqqff] (8.,2.) circle (1.5pt); 
	\draw [color=qqqqff] (8.,3.) circle (1.5pt); 
	\draw [color=qqqqff] (7.,3.) circle (1.5pt); 
	\draw [color=qqqqff] (6.,3.) circle (1.5pt); 
	\draw [color=qqqqff] (5.,3.) circle (1.5pt); 
	\draw [color=qqqqff] (5.,4.) circle (1.5pt); 
	\draw [color=qqqqff] (6.,4.) circle (1.5pt); 
	\draw [color=qqqqff] (7.,4.) circle (1.5pt); 
	\draw [color=qqqqff] (8.,4.) circle (1.5pt); 
	\draw [fill=xdxdff] (5.5,3.) circle (2.5pt); 
	\draw [fill=xdxdff] (6.,3.5) circle (2.5pt); 
	\draw [fill=xdxdff] (6.5,3.) circle (2.5pt); 
	\draw [fill=xdxdff] (6.,2.5) circle (2.5pt); 
	\draw [fill=xdxdff] (5.5,2.) circle (2.5pt); 
	\draw [fill=xdxdff] (5.5,1.) circle (2.5pt); 
	\draw [fill=xdxdff] (6.,1.5) circle (2.5pt); 
	\draw [fill=xdxdff] (6.5,1.) circle (2.5pt); 
	\draw [fill=xdxdff] (7.,1.5) circle (2.5pt); 
	\draw [fill=xdxdff] (6.5,2.) circle (2.5pt); 
	\draw [fill=xdxdff] (7.,2.5) circle (2.5pt); 
	\draw [fill=xdxdff] (7.5,2.) circle (2.5pt); 
	\draw [fill=xdxdff] (8.,1.5) circle (2.5pt); 
	\draw [fill=xdxdff] (7.5,1.) circle (2.5pt); 
	\draw [fill=xdxdff] (7.5,3.) circle (2.5pt); 
	\draw [fill=xdxdff] (8.,2.5) circle (2.5pt); 
	\draw [fill=xdxdff] (8.,3.5) circle (2.5pt); 
	\draw [fill=xdxdff] (7.5,4.) circle (2.5pt); 
	\draw [fill=xdxdff] (7.,3.5) circle (2.5pt); 
	\draw [fill=xdxdff] (5.5,4.) circle (2.5pt); 
	\draw [fill=xdxdff] (5.,3.5) circle (2.5pt); 
	\draw [fill=xdxdff] (6.5,4.) circle (2.5pt); 
	\draw [fill=xdxdff] (5.,2.5) circle (2.5pt); 
	\draw [fill=xdxdff] (5.,1.5) circle (2.5pt); 
\end{scriptsize} 
\end{tikzpicture}
\caption{\label{fig:kite}$\mathbb{Z}^{2}$ and the kite graph}

\end{center}
\end{figure}

We have the following relation:
\[
\eqref{eq:DefGenArete}=\mathcal{L}'_{p}\partial_{b'}f-4\partial_{b}f
\]
So we obtain the following relation:
\begin{align*}
\partial_{b}\mathcal{L}_{e}f & =\left(\mathcal{L}_{e}+\mathcal{L}'_{p}-4Id\right)\partial_{b}f\\
 & =\left(L'-4Id\right)F\left(b,\cdot\right)\\
\Rightarrow\partial_{b}P_{t}f & =e^{4t}P_{t}^{L'}F\left(b,\cdot\right)
\end{align*}
Where $F\left(b,\cdot\right)=\partial_{b}f$. Before stating the next
proposition, we define $\left(X'\left(t\right)\right)_{t}$ the simple
random walk on the oriented kite graph of generator $\mathcal{L}_{p}'$.
A result obtained in \cite{deBuyerThese}, where the proof can be
found, is the following:
\begin{prop}
On $\mathbb{Z}^{2}$ with a Gaussian potential, for any oriented edge
$b\in B$, we have the following relation:
\[
\mathbb{E}\left[\frac{1}{2}\eta_{x}^{2}\right]=e^{4t}\left(\mathbb{P}_{b}\left(X'\left(t\right)=b\right)-\mathbb{P}_{b}\left(X'\left(t\right)=\overleftarrow{b}\right)\right)
\]
\end{prop}

The result can be generalized to a more general class of potential,
i.e. the potential that can be written $H\left(\eta\right)=\sum V_{0}\left(\eta_{x}\right)$
where $V_{0}\in C^{2}\left(\mathbb{R},\mathbb{R}\right)$, $V''_{0}\geqslant C_{-}\in\mathbb{R}$
and $\lim_{\left|t\right|\to\infty}V_{0}\left(t\right)=+\infty$.
However, one can see that this quantity is hard to exploit since we
need to compensate the exponential term.

\bibliographystyle{plain}
\bibliography{bibtex_these-Ginzburg}

\end{document}